\documentclass[a4paper,11pt]{article}
\usepackage[T2A]{fontenc}
\usepackage[cp1251]{inputenc}
\usepackage[ukrainian, english]{babel}
\usepackage{amssymb,amsfonts,amsthm,amsmath,mathtext,cite,enumerate,float}
\usepackage[onehalfspacing]{setspace}
\usepackage{geometry}
\usepackage{graphicx}
\usepackage{multirow}
\usepackage{indentfirst}
\usepackage{color}

\newtheorem{theorem}{Theorem}

\newtheorem{corollary}[theorem]{Corollary}
\newtheorem{lemma}{Lemma}
\newtheorem{example}{Example}

\newtheorem{proper}{Property}
\theoremstyle{definition}

\begin{document}
%\title{Commutator of Sylow 2-subgroups of alternating groups and symmetric group and their description}
%\title{Commutator and description of Sylow 2-subgroups of alternating groups and symmetric group}
\title{THE INVESTIGATION OF EULER'S TOTIENT FUNCTION  PREIMAGES }
\author{ Ruslan Skuratovskii \\ ruslan@unicyb.kiev.ua, ruslcomp@mail.ru \\ Kiev, professor of MAUP, faculty of computer sciences} %,

\date{\empty}
\maketitle
%\date{}
%\maketitle

%\begin{center}
%\parbox{11.5cm}
%{\bf {Description and commutator
%} \/}

 %\begin{section}{Abstract}
 \section*{Abstract}
We propose a lower bound for computing quantity of the inverses of Euler's function. We answer the question about the  multiplicity of $m$ in the equation $\varphi(x) = m$ \cite{Ford}. An analytic expression for exact multiplicity of $m = {{2}^{{{2}^{n}}+a}}$, where $a\in N$, $a<{{2}^{n}}$, $\varphi(t)={{2}^{{{2}^{n}}+a}}$ was obtained. A lower bound of inverses number for arbitrary $m$ was found. We make an approach to Sierpinski assertion from new side. New numerical metric was proposed.\\
\textbf{Key words}: Euler's totient function, inverses of Euler's function, numerical metric, Fermat Numbers, Chen’s Theorem, number of preimages of Euler’s totient function, Sierpinski assertion.
% \end{section}{}

 \begin{section}{Introduction}
The purpose of  this  work is to study theoretical numerical properties  a  multivalued  function \cite{Luca, Rodney } which is inversed to Euler's function, show the relevance of the examples.
% Also we want to explore the composition of this function.

Subject of study: explore the composition of the function $\varphi \left( n \right)$ with itself and the tasks associated with it, it's properties,  the number of preimages of the function $\varphi \left( n \right)$, behavior of the straight $O{{A}_{n}}$, where ${{A}_{n}}\left( n;\varphi \left( n \right) \right)$ and $O\left( 0;0 \right)$ where $n\to \infty $.

We propose a lower estimation for computing quantity of the inverses of Euler's function.
Our approach can be further adapted for computing certain functions of the inverses, such as their quantity, the larger.

Of fundamental importance in the theory of numbers is Euler's totient function $\varphi (n)$. Two famous unsolved problems concern the possible values of the function $A(m)$, the number of solutions of $\varphi(x) = m$, also called the multiplicity of $m$.

Of big importance in the cryptography has number of preimages of Euler's totient function $\varphi (n), \ n=pq$. Because it determines cardinal of secret key space in $RSA$ \cite{RSA}.
  %Two famous unsolved problems concern the possible values of the function $A(m)$, the number of solutions of $\varphi(x) = m$, %also called the multiplicity \cite{Ford} of $m$.

 \end{section}

 \begin{section}{ Number of values of $\varphi^{-1}(m)$ for special classes of numbers}

We propose a exact formula for computing quantity of the inverses of Euler's function for any number of form $2^s$.

An old conjecture of Sierpinski asserts that for every integer $k > 2$, there is a number m for which the equation  $\varphi (t)=m$  has exactly $k$  solutions the number of solutions $t$  of $\varphi (t)=m$, also called the multiplicity of $m$. In this section we find multiplicity for numbers of form $2^s$.

\begin{example} The set of preimages for $12$ is following: ${{\phi }^{-1}}(12)\text{ }=\text{ }\{13,21,26,28,36,42\}$. Also we have    ${{\varphi }^{-1}}\left( 16 \right)=\{32,\text{ }48,\text{ }17,\text{ }34,\text{ }40,\text{ }60\}$,
 ${{\varphi }^{-1}}\left( 18 \right)=\{19,\text{ }27,\text{ }38,\text{ }54\}$.
 We remind, that the number of a form ${{2}^{{{2}^{n}}}}+1$, where $n$ is not-negative integer,  is called Fermat number.
\end{example}

Also the recursive formula for Fermat numbers \cite{Luca, SkVor, Sk, Vinogradov}  was used: ${{F}_{n}}={{F}_{0}}...{{F}_{n-1}}+2$.
Useful for the study of the number of prototypes is Lucas's Theorem:  each prime divisor of the Fermat number  $F_n $, where $n>1$,  has a form of  $k{{2}^{n+2}}+1$.

\begin{lemma}
If ${{2}^{m}}+1$ is prime, then $m={{2}^{n}}$.
\end{lemma}
\begin{proof}
We will prove by contradiction. Suppose there exists a number of a form  ${{2}^{m}}+1$ which is not prime and  $m$ is divisible by $p\ne 2$. Since $p$  is prime and it is not 2, it must be odd. Let $m=pt$, so we can rewrite our number like this: ${{2}^{m}}+1={{\left( {{2}^{t}} \right)}^{p}}+{{\left( 1 \right)}^{p}}=\left( {{2}^{t}}+1 \right)\left( {{\left( {{2}^{t}} \right)}^{p-1}}-...+{{\left( 1 \right)}^{p-1}} \right)$. Expressions in both brackets are grater than 1, but our number is supposed to be prime. Contradiction.
\end{proof}
We make of use Theorem about mutually coprimality of non-prime Fermat number \cite{Vinogradov}.

\begin{theorem}  Let $n\in N \cup\{ 0 \}$. If ${{2}^{{{2}^{n}}}}+1$ is not prime, then for any number of the form ${{2}^{{{2}^{n}}+a}}$, where $a\in N$, $a<{{2}^{n}}$, there exists exactly ${{2}^{t}}$ natural numbers $m$ such that $\varphi \left( m \right)={{2}^{{{2}^{n+a}}}}$, where  $t$ is amount of  prime Fermat numbers, which are less than ${{2}^{{{2}^{n}}}}+1$.
\end{theorem}
\begin{proof}
Consider a set $\left\{ {{p}_{1}},{{p}_{2}},...{{p}_{t}} \right\}$ of all prime Fermat numbers lesser than ${{2}^{{{2}^{n}}}}+1$. Let $\varphi \left( x \right)={{2}^{{{2}^{n}}+a}}$.
According to  Lemma 1, $x={{2}^{s}}{{q}_{1}}{{q}_{2}}...{{q}_{v}}$, where ${{q}_{i}}$ are different prime Fermat numbers.  Since $a<{{2}^{n}}$, then
${{2}^{{{2}^{n}}+a}}<{{2}^{{{2}^{n+1}}}}$. That means, that ${{q}_{i}}<{{2}^{{{2}^{n+1}}}}+1$, because  $\varphi \left( x \right)=\varphi \left( {{2}^{s}}{{q}_{1}}{{q}_{2}}...{{q}_{v}} \right)={{2}^{{{2}^{n}}+a}}<{{2}^{{{2}^{n+1}}}}$.

We also know that ${{q}_{i}}\ne {{2}^{{{2}^{n}}}}+1$, because ${{2}^{{{2}^{n}}}}+1$ is not prime. This yields ${{q}_{i}}<{{2}^{{{2}^{n}}}}+1$.
%In other words, ${{q}_{i}}\in \left\{ {{p}_{1}},{{p}_{2}},...{{p}_{t}} \right\}$.
 Other words it can be written like this: $\left\{ {{q}_{1}},{{q}_{2}},...{{q}_{v}} \right\}\subseteq \left\{ {{p}_{1}},{{p}_{2}},...{{p}_{t}} \right\}$. For each $x$ we get, that
$\left\{ {{q}_{1}},{{q}_{2}},...{{q}_{v}} \right\}$ is a subset of  the set $\left\{ {{p}_{1}},{{p}_{2}},...{{p}_{t}} \right\}$. We shall prove, that each subset of the set ${{M}_{t}}=\left\{ {{p}_{1}},{{p}_{2}},...{{p}_{t}} \right\}$ determines such unique $x$ as a unique product of this subset of primes from ${{M}_{t}}$, that $x$ with a corresponding  multiplier ${{2}^{s}},\,\,s\in N\cup \{0\}$ gives us  $x={{2}^{s}}t$  such   that  $\varphi \left( x \right)={{2}^{{{2}^{n}}+a}}$.

 For this goal we need to show, that $\varphi \left( {{p}_{1}}\cdot {{p}_{2}}\cdot ...{{p}_{t}} \right)<{{2}^{{{2}^{n}}+a}}$.

 Since $\varphi \left( {{p}_{1}}\cdot {{p}_{2}}\cdot ...{{p}_{t}} \right)$ is Euler's function of a product of prime Fermat numbers, which lesser than ${{2}^{{{2}^{n}}}}+1$, it is not grater than value of Euler's function of a product  of all Fermat numbers, which lesser than ${{2}^{{{2}^{n}}}}+1$, which is equal to \[\varphi \left( \left( {{2}^{{{2}^{0}}}}+1 \right)...\left( {{2}^{{{2}^{n-1}}}}+1 \right) \right). \]

 That is true, as obvious inequality holds: $\varphi \left( d \right)\le \varphi \left( db \right)$. It is also known, that any two Fermat numbers are coprime \cite{Vinogradov}, so \[\varphi \left( \left( {{2}^{{{2}^{0}}}}+1 \right)...\left( {{2}^{{{2}^{n-1}}}}+1 \right) \right)=\varphi \left( {{2}^{{{2}^{0}}}}+1 \right)...\varphi \left( {{2}^{{{2}^{n-1}}}}+1 \right).\]

As known, $\varphi \left( y \right)\le y-1$, therefore
\[\varphi \left( {{2}^{{{2}^{0}}}}+1 \right)...\varphi \left( {{2}^{{{2}^{n-1}}}}+1 \right)\le \left( {{2}^{{{2}^{0}}}}+1-1 \right)\cdot ...\cdot \left( {{2}^{{{2}^{n-1}}}}+1-1 \right)={{2}^{0}}\cdot ...\cdot {{2}^{{{2}^{n-1}}}}={{2}^{{{2}^{n}}-1}}.\]
It was used the formula of the sum of geometric progression,  we have ${{2}^{0}}+{{2}^{1}}+...+{{2}^{n-1}}={{2}^{n}}-1$. Therefore
$\left( {{2}^{{{2}^{0}}}}+1-1 \right)\cdot ..\cdot \left( {{2}^{{{2}^{n-1}}}}+1-1 \right)={{2}^{{{2}^{0}}+{{2}^{1}}+..+{{2}^{n-1}}}}={{2}^{{{2}^{n}}-1}}$.

Finally, $$\varphi \left( {{p}_{1}}\cdot {{p}_{2}}\cdot ...{{p}_{t}} \right)\le {{2}^{{{2}^{n}}-1}}<{{2}^{{{2}^{n}}+a}},$$
what was needed. That means, that Euler's function of the product of the elements of any subset of the set $\left\{ {{p}_{1}},{{p}_{2}},...{{p}_{t}} \right\}$ is lesser than ${{2}^{{{2}^{n}}+a}}$.
Let us take an arbitrary subset of $\left\{ {{p}_{1}},{{p}_{2}},...{{p}_{t}} \right\}$. Let the elements of this set be $\left\{ {{q}_{1}},{{q}_{2}},...{{q}_{v}} \right\}$. Consider the expression $\varphi \left( {{q}_{1}}\cdot {{q}_{2}}\cdot ...\cdot {{q}_{v}} \right)={{2}^{w}}<{{2}^{{{2}^{n}}+a}}$.
 This inequality means, that we can choose such natural number $s$, so $\varphi \left( {{2}^{s}}\cdot {{q}_{1}}\cdot {{q}_{2}}\cdot ...\cdot {{q}_{v}} \right)={{2}^{s-1}}\cdot {{2}^{w}}={{2}^{{{2}^{n}}+a}}$. In other words, for given subset $\left\{ {{q}_{1}},{{q}_{2}},...{{q}_{v}} \right\}$, we found such number $x$, that $\varphi \left( x \right)={{2}^{{{2}^{n}}+a}}$. The last equality means, that each subset defines unique $x$.
%Let us take an arbitrary subset of the set $\left\{ {{p}_{1}},{{p}_{2}},...{{p}_{t}} \right\}$. Let the elements of this set be $\left\{ {{r}_{1}},{{r}_{2}},...{{r}_{u}} \right\}$. Consider the expression $\varphi \left( {{r}_{1}}\cdot {{r}_{2}}\cdot ...\cdot {{r}_{u}} \right)={{2}^{w}}<{{2}^{{{2}^{n}}+a}}$.
% This inequality means, that we can find such natural number $s$, so $\varphi \left( {{2}^{s}}\cdot {{r}_{1}}\cdot {{r}_{2}}\cdot ...\cdot {{r}_{u}} \right)={{2}^{s-1}}\cdot {{2}^{w}}={{2}^{{{2}^{n}}+a}}$. In other words, for given subset $\left\{ {{r}_{1}},{{r}_{2}},...{{r}_{u}} \right\}$, we found such number $x$, that $\varphi \left( x \right)={{2}^{{{2}^{n}}+a}}$. The last thing to say is, that each subset defines unique  $x$.
 Therefore, each subset gives us the needed the number $x$ that is always determined by some subset. In other words, the amount of needed numbers is exactly the amount of different possible subsets. As well-known fact, this amount is equal ${{2}^{t}}$ for a set of $t$  elements.
\end{proof}

     \begin{example} For a non-prime Fermat number ${{2}^{32}}+1,$ number of preimages  for subsequent numbers  of the form ${{2}^{{{2}^{n}}+a}},\,\,\,\,a\le 32-1, \,\,\,  n \le 4$  is  equal  to  $2^{32}$.

For generalizing of Theorem 2 it is easy to prove the following statement:

\begin{theorem}
 Let $a\in Z,\,\,  0\le a\le {{2}^{n}}$, then  the number of solutions of  $\varphi (x)={{2}^{{{2}^{n}}+a}}$ is equal to the number of sets $\{{{2}^{{{i}_{1}}}},\,...,\,\,{{2}^{{{i}_{k}}}}\}$, such that:
%$\left\{
\begin{align*}
  & {{i}_{1}}<{{i}_{2}}<\,\,...\,\,<{{i}_{k}} \\
 & {{2}^{{{i}_{1}}}}+{{2}^{{{i}_{2}}}}+...+{{2}^{{{i}_{k}}}}\le {{2}^{n}}+a \\
 & {{2}^{{{2}^{{{i}_{1}}}}}}+{{1,...,2}^{{{2}^{{{i}_{k}}}}}}+1\in {{\text{F}}_{pr}}, \\
\end{align*}
% \right\}$, 				
where ${{\text{F}}_{pr}}$ is a  set of  Ferma's prime numbers.

If  ${{2}^{{{2}^{n}}}}+1$ is not prime, then the number of specified sets (including  empty set) is equal to ${{2}^{t}}$,  where $t$ is a number of  Ferma's  prime numbers smaller then ${{2}^{{{2}^{n}}}}+1$.
\end{theorem}

\begin{proof}
To construct the necessary preimage $x$ over the set of  Ferma's  primes with the properties of this Theorem $\varphi (x)={{2}^{{{2}^{n}}+a}}$  we proceed as follows:

1) We choose a combination of this numbers. Let us call it \[{{\left( {{2}^{{{2}^{{{i}_{0}}}}}}+1 \right)}}...{{\left( {{2}^{{{2}^{{{i}_{k-1}}}}}}+1 \right)}}.\]

2) Then we should find its total power of 2 that is ${{2}^{{{i}_{1}}}}+{{2}^{{{i}_{2}}}}+...+{{2}^{{{i}_{k}}}}=s$, this power be obtained after calculating the Euler's function from the product $\varphi \left( {{\left( {{2}^{{{2}^{{{i}_{0}}}}}}+1 \right)}}...{{\left( {{2}^{{{2}^{{{i}_{k}}}}}}+1 \right)}} \right)$  and also satisfies the inequality

$$s={{2}^{{{i}_{1}}}}+{{2}^{{{i}_{2}}}}+...+{{2}^{{{i}_{k}}}}\le {{2}^{n}}+a.$$

We supplement the received power exponent $s$ to the necessary ${{2}^{n}}+a$ by multiplying the product of
\[{{\left( {{2}^{{{2}^{{{i}_{0}}}}}}+1 \right)}}...{{\left( {{2}^{{{2}^{{{i}_{k}}}}}}+1 \right)}}\]

    on $2^{{{2}^{n}}+a-s}$. Thus, the necessary preimage $x$ is constructed.
\end{proof}

\begin{proper}
For any number $S$ of the form $p_{1}^{{{\alpha }_{1}}}p_{2}^{{{\alpha }_{2}}}...\,p_{k}^{{{\alpha }_{k}}},\,\,{{p}_{1}}>2$,  where ${{p}_{1}},\,{{p}_{2}},\,...\,,{{p}_{k}}$ are odd prime numbers,  the following equality  holds: $\varphi (S)=\varphi (2S).$
\end{proper}

\begin{proof}
Since  2 and $p_{1}^{{{\alpha }_{1}}}p_{2}^{{{\alpha }_{2}}}...\,p_{k}^{{{\alpha }_{k}}},\,\,{{p}_{1}}>2$ are coprime,  then \[\varphi \left( 2p_{1}^{{{\alpha }_{1}}}p_{2}^{{{\alpha }_{2}}}...\,p_{k}^{{{\alpha }_{k}}} \right)=\varphi \left( 2 \right)\varphi \left( p_{1}^{{{\alpha }_{1}}}p_{2}^{{{\alpha }_{2}}}...\,p_{k}^{{{\alpha }_{k}}} \right)=\varphi \left( p_{1}^{{{\alpha }_{1}}}p_{2}^{{{\alpha }_{2}}}...\,p_{k}^{{{\alpha }_{k}}} \right).\]
 Therefore these numbers has the same of Euler's function.
\end{proof}
% If for the set $\left\{ {{p}_{1}},\,{{p}_{2}},\,...\,,{{p}_{k}} \right\},\,\,{{p}_{1}}>2$,  then for the number  that has in  canonical  presentation only factors from
%$p_{1}^{{{\alpha }_{1}}}p_{2}^{{{\alpha }_{2}}}...\,p_{k}^{{{\alpha }_{k}}},\,\,{{p}_{1}}>2$  there exists two preimages $p_{1}^{{{\alpha }_{1}}}p_{2}^{{{\alpha }_{2}}}...\,p_{k}^{{{\alpha }_{k}}}$   and $2p_{1}^{{{\alpha }_{1}}}p_{2}^{{{\alpha }_{2}}}...\,p_{k}^{{{\alpha }_{k}}}$.
\end{example}

 \begin{theorem}   If $\varphi \left( m \right)={{2}^{n}}$, then $m={{2}^{s}}{{p}_{1}}{{p}_{2}}...{{p}_{x}}$, where ${{p}_{i}}$ are different Fermat numbers, $s\in N$.
\begin{proof}
Firstly, we will prove that $m$ can be divisible by odd prime number $p$, only if $ p$ is prime Fermat number.
Let's consider an arbitrary prime number p such that $m\vdots p$. Then $m={{p}_{1}}^{{{\text{ }\!\!\alpha\!\!\text{ }}_{1}}}{{p}_{2}}^{{{\text{ }\!\!\alpha\!\!\text{ }}_{2}}}...{{p}_{x}}^{{{\text{ }\!\!\alpha\!\!\text{ }}_{x}}}\cdot {{p}^{\text{ }\!\!\alpha\!\!\text{ }}}$, where $\text{ }\!\!\alpha\!\!\text{ }\ge 1$.
As Euler's function is multiplicative and  $\varphi \left( {{p}^{\text{x}}} \right)={{p}^{\text{x-1}}}(p-1)$, as well-known fact, we can write:

$\varphi \left( m \right)=\varphi \left( {{p}_{1}}^{{{\text{ }\!\!\alpha\!\!\text{ }}_{1}}}{{p}_{2}}^{{{\text{ }\!\!\alpha\!\!\text{ }}_{2}}}...{{p}_{x}}^{{{\text{ }\!\!\alpha\!\!\text{ }}_{x}}} \right)\cdot \varphi \left( {{p}^{\text{ }\!\!\alpha\!\!\text{ }}} \right)=\varphi \left( {{p}_{1}}^{{{\text{ }\!\!\alpha\!\!\text{ }}_{1}}}{{p}_{2}}^{{{\text{ }\!\!\alpha\!\!\text{ }}_{2}}}...{{p}_{x}}^{{{\text{ }\!\!\alpha\!\!\text{ }}_{x}}} \right)\cdot {{p}^{\text{ }\!\!\alpha\!\!\text{ -1}}}(p-1)$.

Therefore, $\varphi \left( m \right)\vdots \left( p-1 \right)$. If  $(p-1)$ is divisible by odd prime number $q$, then $\varphi \left( m \right)$ is also divisible by $q$, which can't be, because $\varphi \left( m \right)={{2}^{n}}$. That is why $ p-1$ can be divisible only by 2, which yields $(p-1)={{2}^{k}}$, or
$p={{2}^{k}}+1$. According to lemma, $k={{2}^{n}}$, so we can rewrite $p={{2}^{{{2}^{n}}}}+1$. In other words, $p$ is Fermat prime number, what was needed.

Secondly, we will prove that $m$  can't be divisible by ${{p}^{2}}$, where $p$ is prime, $p\ne 2$. We will prove by contradiction.
 If $m$ is divisible by ${{p}^{2}}$, then $m={{p}_{1}}^{{{\text{ }\!\!\alpha\!\!\text{ }}_{1}}}{{p}_{2}}^{{{\text{ }\!\!\alpha\!\!\text{ }}_{2}}}...{{p}_{x}}^{{{\text{ }\!\!\alpha\!\!\text{ }}_{x}}}\cdot {{p}^{\text{ }\!\!\alpha\!\!\text{ }}}$, where  $\text{ }\!\!\alpha\!\!\text{ }\ge 2$. As we already know, \[\varphi \left( m \right)=\varphi \left( {{p}_{1}}^{{{\text{ }\!\!\alpha\!\!\text{ }}_{1}}}{{p}_{2}}^{{{\text{ }\!\!\alpha\!\!\text{ }}_{2}}}...{{p}_{x}}^{{{\text{ }\!\!\alpha\!\!\text{ }}_{x}}} \right)\cdot \varphi \left( {{p}^{\text{ }\!\!\alpha\!\!\text{ }}} \right)=\varphi \left( {{p}_{1}}^{{{\text{ }\!\!\alpha\!\!\text{ }}_{1}}}{{p}_{2}}^{{{\text{ }\!\!\alpha\!\!\text{ }}_{2}}}...{{p}_{x}}^{{{\text{ }\!\!\alpha\!\!\text{ }}_{x}}} \right)\cdot {{p}^{\text{ }\!\!\alpha\!\!\text{ -1}}}(p-1) ,\]  so $\varphi \left( m \right)\vdots p$. But $p\ne 2$, while $\varphi \left( m \right)={{2}^{n}}$. Contradiction.
\end{proof}
\end{theorem}

\end{section}

\begin{section}{Low bound of $\varphi^{-1}(m)$ values for number of general form}

We propose a lower bound for computing quantity of the inverses of Euler's function.
Our approach can be further adapted for computing certain functions of the inverses, such as their quantity, the larger.

Let ${{M}_{k}}$ be a set of first $k$ consecutive primes.
We will say, that the number is \emph{decomposed over a set} ${{M}_{k}}$, if in its canonical decomposition there are only numbers from ${{M}_{k}}$.
Let ${{x}_{1}},\,...\,,\,{{x}_{n+2}}$  be such numbers, that $\varphi \left( {{x}_{1}} \right)=\,\varphi \left( {{x}_{2}} \right)=...\,=\,\varphi \left( {{x}_{n+2}} \right)$, and at the same time all prime factors of the canonical decomposition   belong to the set ${{M}_{n}}=\{{{p}_{0}},....,{{p}_{n}}\}$, where ${{p}_{0}}=2$ and ${{p}_{i}}$ are all consecutive prime numbers.  Let for any natural number $n$, we have ${{Q}_{n}}=\left( {{p}_{0}}-1 \right)\left( {{p}_{1}}-1 \right)...\left( {{p}_{n-1}}-1 \right)\left( {{p}_{n}}-1 \right)$,
where ${{p}_{i}}$ is $i $-th odd prime number, where $i\in N$ and ${{p}_{0}}=2$.
Example: ${{p}_{1}}=3,\text{ }{{p}_{2}}=5,\text{ }{{p}_{3}}=7$,  then  ${{Q}_{3}}=\left( {{p}_{0}}-1 \right)\left( {{p}_{1}}-1 \right)\left( {{p}_{2}}-1 \right)\left( {{p}_{3}}-1 \right)  = \left( 3-1 \right)\left( 5-1 \right)\left( 7-1 \right)=48$.

Let $M_k$ be a set of $k$ consequent first prime numbers.
The following statement about estimation of preimages number is true.
\begin{theorem}

 For any natural number $n$  there are such various natural numbers \\ ${{x}_{1}}\text{, }{{x}_{2}},\text{ }...\text{ }\text{,}\,{{x}_{n}}\text{, }{{x}_{n+1}}\text{, }{{x}_{n+2}}$, that
$$\varphi \left( {{x}_{1}} \right)=\varphi \left( {{x}_{2}} \right)=...=\varphi \left( {{x}_{n+2}} \right)={{Q}_{n}},$$
where every number ${{x}_{i}}$ contains in its canonical decomposition \cite{Vinogradov} only ${{p}_{i}}$ from ${{M}_{n}}$, and \[{{x}_{n+2}}={{p}_{0}}{{p}_{1}}\,...\,{{p}_{n-1}}{{p}_{n}}\] holds.

\end{theorem}

\begin{proof}
We prove it by mathematical induction.

Base case: given $n = 1$, then ${{P}_{1}}=\left( {{p}_{1}}-1 \right)=2$  has at least three prototypes. This statement is true, because $\varphi \left( 3 \right)=\varphi \left( 4 \right)=\varphi \left( 6 \right)=2$. Base case is proved.

Step case: if for $n=k$ it holds, we will prove, that for $n=k+1$ it holds too. By the assumption we have, that for natural number $n$ were found such various natural ${{x}_{1}}\text{, }{{x}_{2}},\text{ }...\text{ }\text{, }{{x}_{k+1}}\text{, }{{x}_{k+2}}$,  that

\begin{equation} \label{Equl}
\varphi \left( {{x}_{1}} \right)=\varphi \left( {{x}_{2}} \right)=...=\varphi \left( {{x}_{k+1}} \right)=\varphi \left( {{x}_{k+2}} \right)={{Q}_{k}}=Q,
\end{equation}

where
${{Q}_{k+1}}=p_{0}^{{{\beta }_{0}}}p_{1}^{{{\beta }_{1}}}\,...\,p_{k}^{{{\beta }_{k}}},$

\[{{x}_{k+1}}={{p}_{1}}{{p}_{2}}...{{p}_{k-1}}{{p}_{k}}, \, {{x}_{k+2}}={{p}_{0}}{{p}_{1}}{{p}_{2}}...{{p}_{k-1}}{{p}_{k}}.\]
Let us make induction transition.
   Prove, that for  $n=k+1$ exist such various natural ${{y}_{1}}\text{, }{{y}_{2}}\text{, }...\text{ }\text{, }{{y}_{k+2}},\text{ }{{y}_{k+3}}$, for which holds:

\begin{equation} \label{Equl}
 \varphi \left( {{y}_{1}} \right)=\varphi \left( {{y}_{2}} \right)=\,\,...\,\,=\varphi \left( {{y}_{k+2}} \right)=\varphi \left( {{y}_{k+3}} \right)={Q}_{k+1},
\end{equation}
each of which has a canonical decomposition over ${{M}_{k}}\cup p_{k+3}$.
% and following prime ${{p}_{k+3}}.$
Clear, that  $\varphi \left( {{p}_{k+3}} \right)$ has a canonical decomposition into elements of ${{M}_{k}}$.

Therefore, it can be presented as $$\varphi \left( {{p}_{k+3}} \right)=p_{0}^{{{\beta }_{0}}}p_{1}^{{{\beta }_{1}}}\, ... \,p_{k}^{{{\beta }_{k}}}.$$

Let's construct new numbers  ${{y}_{1}}\text{, }{{y}_{2}}\text{, }\,...\text{ }\,\text{,}{{y}_{k+1}}\text{, }{{y}_{k+2}}\text{, }{{y}_{k+3}}$  such that:

\[{{y}_{1}}\text{=}\,{{x}_{1}}{{p}_{k+3}}\text{, }\,\,{{y}_{2}}\text{=}\,{{x}_{2}}{{p}_{k+3}}\text{, }\,...\text{ }\,\text{,}{{y}_{k+1}}\text{=}\,{{x}_{k+1}}{{p}_{k+3}}\text{,}\,\text{ }{{y}_{k+2}}\text{=}\,{{x}_{k+2}}{{p}_{k+3}}\]

 In this case, the value of the Euler function is ${{Q}_{k+1}}=p_{0}^{{{\beta }_{0}}}p_{1}^{{{\beta }_{1}}}\,...\,p_{k}^{{{\beta }_{k}}}.$
  Let us show, that all  ${{y}_{1}}\text{, }\,...\text{ }\,\text{,}{{y}_{k+2}},{{y}_{k+3}}$ are different.

 Since numbers ${{x}_{1}}\text{, }{{x}_{2}},\text{ }...\text{ }\text{, }{{x}_{k+1}}\text{, }{{x}_{k+2}}$  from (1) have different canonical decompositions, so the decompositions of numbers ${{y}_{1}}\text{, }{{y}_{2}}\text{, }\,...\text{ }\,\text{,}{{y}_{k+1}}\text{, }{{y}_{k+2}}$  over $M_k$  are different too, but they all have a new factor  ${{p}_{k+1}}$, but do not decompose over ${{M}_{k}}$. A last one ${{y}_{k+3}}$ also decomposes over ${{M}_{k}}$  and does not contain a factor  ${{p}_{k+1}}$.
But value   ${{Q}_{k+1}}=p_{0}^{{{\beta }_{0}}}p_{1}^{{{\beta }_{1}}}\,...\,p_{n}^{{{\beta }_{n}}}$  does not contain ${{p}_{k+1}}$ in the decomposition, so there is at least one number ${{y}_{k+3}}$ with decomposition over ${{M}_{k}}$, such, that $\varphi \left( {{y}_{k+3}} \right)={{Q}_{k+1}}$ holds.

Since ${{Q}_{k+1}}>{{Q}_{k}}$, then a new preimage ${{y}_{k+3}}$ does not coincide with any of the numbers ${{y}_{1}}\text{, }{{y}_{2}},\text{ }...\text{ }\text{, }{{y}_{k+1}}\text{, }{{y}_{k+2}}$ which give the value of Euler's function equal ${{Q}_{k}}$.

Moreover such ${{y}_{k+3}}$ can be not unique number that can be constructed over ${{M}_{k}}$ such, that $\varphi \left( {{y}_{k+3}} \right)={{Q}_{k+1}}$. Consequently beyond  ${{y}_{1}}\text{, }{{y}_{2}}\text{, }\,...\text{ }\,\text{,}{{y}_{k+1}}\text{, }{{y}_{k+2}}$, which decomposed over ${{M}_{k+1}}$, we  have at least   one new ${{y}_{k+3}}$,  which can be decomposed over ${{M}_{k}}$ in product of primes. Thus ${{Q}_{k+1}}$ has at least $k+3$ different preimages.
\end{proof}

We propose \textbf{method of constructing} of such preimages set.

Let ${{p}_{0}}=2,{{p}_{1}}=3,{{p}_{2}}=5,\ldots ,{{p}_{n}}$ be consecutive prime numbers.
Note, that \\
 $\varphi ({{p}_{0}}{{p}_{1}},\ldots ,{{p}_{n}})=({{p}_{0}}-1)({{p}_{1}}-1)...({{p}_{n}}-1).$
Let's construct some other numbers ${{x}_{0}},\ldots ,{{x}_{n}}$, for which
\[\varphi ({{x}_{0}})=\varphi ({{x}_{1}})=...=\varphi ({{x}_{n}})=\varphi ({{p}_{0}},{{p}_{1}},\ldots ,{{p}_{n}})=({{p}_{0}}-1)({{p}_{1}}-1)...({{p}_{n}}-1).\]
Namely, let
\begin{align*}
  & {{x}_{0}}=({{p}_{0}}-1){{p}_{0}},\ldots ,{{p}_{n}} \\
 & {{x}_{1}}={{p}_{0}}({{p}_{1}}-1){{p}_{2}},\ldots ,{{p}_{n}} \\
 & \cdots    \\
  & {{x}_{n}}={{p}_{0}}{{p}_{1}},\ldots ,{{p}_{n-1}}({{p}_{n-1}}-1) \\
\end{align*}

Now let's prove, that $\varphi ({{p}_{0}}{{p}_{1}}...{{p}_{k-1}}({{p}_{k}}-1){{p}_{k+1}}...{{p}_{n}})=({{p}_{0}}-1)({{p}_{1}}-1)...({{p}_{n}}-1)$
for every $k=0,1,...,n$.
Obviously,  ${{p}_{0}}...{{p}_{k-1}}({{p}_{k}}-1)$  and  ${{p}_{k+1}}...{{p}_{n}}$ are  coprime, so  $\varphi ({{x}_{k}})=\varphi ({{p}_{0}}{{p}_{1}}...{{p}_{k-1}}({{p}_{k}}-1))\times \varphi ({{p}_{k+1}}  \ldots {{p}_{n}})=\varphi ({{p}_{0}}{{p}_{1}}...{{p}_{k-1}}({{p}_{k}}-1))\times ({{p}_{k+1}}-1)...({{p}_{k}}-1)$
That is, we have to prove the equality $\varphi ({{p}_{0}}{{p}_{1}}...{{p}_{k-1}}({{p}_{k}}-1))=({{p}_{0}}-1)({{p}_{1}}-1)...({{p}_{k}}-1)$.

Since only ${{p}_{0}}{{p}_{1}},\ldots ,{{p}_{k-1}}$  are the prime numbers, which are not more than  $({{p}_{k}}-1)$,  we have ${{p}_{k}}-1={{\alpha }_{0}}{{\alpha }_{1}},\ldots ,{{\alpha }_{k-1}}$  for some non-negative integer ${{\alpha }_{0}}{{\alpha }_{1}},\ldots ,{{\alpha }_{k-1}}$.

By direct calculation we obtain \begin{align*}
  & \varphi ({{p}_{0}}{{p}_{1}}...{{p}_{k-1}}({{p}_{k}}-1))=\varphi (p_{0}^{{{\alpha }_{0}}+1}p_{1}^{{{\alpha }_{1}}+1}...p_{k-1}^{{{\alpha }_{k-1}}+1})=\\
  &({{p}_{0}}-1)...({{p}_{k-1}}-1)p_{0}^{({{\alpha }_{0}}+1)-1}...p_{k-1}^{({{\alpha }_{k-1}}+1)-1}= \\
 & =({{p}_{0}}-1)({{p}_{1}}-1)...({{p}_{k-1}}-1)p_{0}^{{{\alpha }_{0}}}...p_{k-1}^{{{\alpha }_{k-1}}}=({{p}_{0}}-1)({{p}_{1}}-1)...({{p}_{k-1}}-1)({{p}_{k}}-1) \\
\end{align*}

Also we may subtract 1 from more than one ${{p}_{k}}$,  if $({{p}_{k}}-1)$ has the decomposition into prime factors, which does not contain some  ${{p}_{j}}(j<k)$.
For example, $\varphi ({{p}_{0}}{{p}_{1}}{{p}_{2}}{{p}_{3}})=\varphi (2\times 3\times 5\times 7)=48$. Except  $(2-1)\times 3\times 5\times 7$,  $2\times (3-1)\times 5\times 7$,   $2\times 3\times (5-1)\times 7$  and   $2\times 3\times 5\times (7-1)$, we may take, for example, $2\times (3-1)(5-1)\times 7$, because  $(3-1)=2$  and   $(5-1)={{2}^{2}}$.
Hence  $\varphi (2\times (3-1)(5-1))=(2-1)(3-1)(5-1)$  by the same arguments, as for ${{p}_{0}},\ldots ,{{p}_{k-1}}({{p}_{k}}-1){{p}_{k+1}},\ldots ,{{p}_{n}}$. So, we may construct at most  ${{2}^{n}}$  products of the form  ${{p}_{0}}{{q}_{1}},\ldots ,{{q}_{n}}$, where ${{q}_{k}}={{p}_{k}}$. Also $({{p}_{0}}-1){{p}_{1}},\ldots ,{{p}_{n}}$ fits for the requirement $\varphi ({{p}_{0}}-1){{p}_{1}},\ldots ,{{p}_{n}}=({{p}_{0}}-1)...({{p}_{n}}-1)$, so we have at most ${{2}^{n}}+1$ numbers, which give us the
same meaning of  $\varphi $, as  ${{p}_{0}},\ldots ,{{p}_{n}}$. Note, that it is not necessarily the complete set of such numbers  $x$, for which  $\varphi (x)={{p}_{0}}{{p}_{1}},\ldots ,{{p}_{n}}$, but it is the set, which may be obtained by the given by us scheme.

An old Sierpinski conjecture asserts that for every integer $k > 2$, there is a number m for which the equation  $\varphi (t)=m$  has exactly $k$  solutions the number of solutions $t$  of $\varphi (t)=m$, also called the multiplicity of $m$. In this section we We find multiplicity for numbers of form $m={{2}^{2}}{{3}^{x}}$,  $m={{2}^{k}}{{3}^{n}}$.
denote by ${{\eta }_{p}}(x)$  the $p$-adic power exponent of number $m$ viz such minimal number of $n$ that $m=t^na$, where $a$ is divisor of unit in the ring $Z_p$. % norm of number $x$.
%In general case for value of Euler's function  $\varphi \left( m \right)={{2}^{k}}{{3}^{n}}.$ we have.
\begin{theorem}
  There exists 3 classes of possible preimages of $m$  which satisfy $\varphi \left( t \right)={{2}^{k}}{{3}^{n}}$:
\end{theorem}
\begin{enumerate}
  \item  If $t = {{2}^{y}}{{3}^{x}}$, then  $y\in \{1,...,k+1\}$,  $x\in N\cup \{0\}$.
  \item If $t = {{2}^{y}}{{3}^{z}}{{p}_{1}}$ and  the following conditions holds:

$\left\{ \begin{matrix}
   \left( x-1 \right)+1+~{{\text{ }\!\!\eta\!\!\text{ }}_{2}} \left(  \varphi({{p}_{i}})  \right)=k,  \\
   y-1+~{{\text{ }\!\!\eta\!\!\text{ }}_{3}}\left( \varphi ({{p}_{1}}) \right)=~n.  \\
\end{matrix} \right.$

  \item   If $t = {{2}^{x}}{{3}^{y}}\,\cdot \underset{i=1}{\overset{L}{\mathop{\prod{{{p}_{i}}}}}}\,$  where ${{p}_{i}}\in P$ such that  ${{p}_{i}}={{2}^{{{v_i}_{1}}}}{{3}^{{{v_i}_{2}}}}+1$ in particular $p_i$ can be prime Fermat number. For such numbers the following conditions hold: $x+\sum\limits_{i=1}^{L}{{{\eta }_{2}}(\varphi({{p}_{i}}))}=k$ and       $y-1 +\sum\limits_{i=1}^{L}{{{\eta }_{3}}(\varphi({{p}_{i}}))}= n$.
\end{enumerate}

\begin{proof}
Proof is carried out by direct verification. So the resulting power of two in $m$, from the item 3,  is the following   $(x-1)+1+\sum\limits_{i=1}^{L}{{{\eta }_{2}}(\varphi({{p}_{i}}))}=x+\sum\limits_{i=1}^{L}{{{\eta }_{2}}({{p}_{i}})}=k$. Also resulting power of 3 is such $y-1 +\sum\limits_{i=1}^{L}{{{\eta }_{3}}(\varphi ({{p}_{i}}))}= n$. On account of the set of primes in item 3 consists of number of form $${{p}_{i}}={{2}^{{{u}_{i}}}}{{3}^{{{v}_{i}}}}+1$$ let quantity of such number is $L$  and numbers in form of prime Fermat number $p_i=2^{2^i}+1$ let sum of such number is $F$, then $k-x-F$ can be splited in sum of number of primes of first form and second form. If we now $F$ for concrete $m$, then we express multiplicity of m as $C_{k-F+L-x-1}^{L-1}$.        %and we multiplicate it on $C_{n-M+y-1}^{M+y-1}$. %, viz $L=N+M$

Recall that quantity of splits of form  ${{v}_{1}}+{{v}_{2}}+...+{{v}_{L}}=n-y+1$ can be counted by the formula   $C_{n-y+1+L-1}^{L-1}$. Consider the canonical decomposition of factor $\underset{i=1}{\overset{L}{\mathop{\prod{{{p}_{i}}}}}}\,$ from item 3.
The expression of the exact number of methods of splits of the numbers $n$ and $k$ (which appear in item 3) is reduced to finding the number of pairs of such vectors $({{u}_{L}},....,{{u}_{L}})$ and $({{v}_{1}},....,{{v}_{m}})$ that $\sum\limits_{i=1}^{L}{{{u}_{i}}}=k$ and $\sum\limits_{i=1}^{L}{{{v}_{i}}}=n$ with respect of
the columns of form

% \begin{align*}
%  & {{u}_{i}} \\
% & {{v}_{i}} \\
%  \end{align*}

% \begin{equation} \label{Vect}
%\left(
 %  {{u}_{i}}, \\
 %  {{v}_{i}} \\
%\right)
%\end{equation}

%\begin{centre}
$$
\left(
  \begin{array}{c} \label{Vect}
    {{u}_{i}}  \\
    {{v}_{i}}  \\
  \end{array}
\right)
$$
%\end{centre}
%  & {{u}_{i}} \\
% & {{v}_{i}} \\
 must be pairwise different.
 In this case, in the matrix of degrees 2 and 3 in the preimages, which has the form

% \begin{align}   \left(
%  & {{v}_{1}},....,{{v}_{L}} \\
% & {{u}_{1}},....,{{u}_{L}} \\
%\right)
%\end{align}

 $$
 \left(
   \begin{array}{cccc}
     {v}_{0} & {v}_{1} & \cdots & {v}_{L} \\
     {u}_{0} & {u}_{1} & \cdots & {u}_{L} \\
   \end{array}
 \right)
 $$

where ${u}_{0}=x+F$, ${v}_{0}=y-1$.  Then the formula is calculated

\begin{equation} \label{C}
C_{k+L-F-y-1}^{L-1}, 						
\end{equation}

%\begin{align} \label(C)
%C_{k+L-F-y-1}^{L-1}, 						
%\end{align}

will give us the exact number of partitions. Also, when calculating the
number of combinations (\ref{C}) using the formula, we assume that the matrices are the same, which are obtained from each other by rearranging the columns even if these columns are different. Other words, the number of  splits $C_{k+L-F-y-1}^{L-1}$ considers different partitions where the parts can be rearranged, but in our case the parts are terms  and therefore nothing changes from the permutation.
Therefore, it necessary to divide by factorial $L!$
Thus, we obtain

\begin{equation} \label{CC}
\frac{C_{k+L-F-y-1}^{L-1}}{L!}.			
\end{equation}

%\begin{align}\label(CC)
%\frac{C_{k+L-F-y-1}^{L-1}}{L!}.
%\end{align}

Where $F$ are numbers in form of prime Fermat number $p_i = 2^{2^i}  +1$ let sum of such number is F. This formula counts distributions by all number of form ${{p}_{i}}={{2}^{{{u}_{i}}}}{{3}^{{{v}_{i}}}}+1$  but not all of them are prime, consequently we obtain an upper bound.

The number of methods to choose the exponent $x$  we note by $X$.  Thus we have $\frac{C_{k+L-F-y-1}^{L-1}}{L!}X$.

Also the formula (\ref{C}) distinguish splits with a same values of $(u_i, v_i)$ and $(u_j, v_j)$ where $u_j=u_i$ and $v_j=v_i$.  Then we have to apply exception-inclusion formula. According to this formula we have to subtract a number of intersections of  sets with the same columns of powers $(u_i, v_i)$.

The exact number of partitions is estimated below by the exception inclusion formula. If we stop at negative summand we shall get a a lower bound.
% at the even intersection we obtain low bound. We need to evaluate

 \[I = \left| U-\bigcup\limits_{\begin{smallmatrix}
 i,j=1 \\
 i<j
\end{smallmatrix}}^{L}{{{M}_{ij}}} \right|,  \] where $U$ is the set of all distributions  of   $k-F-y$ into  $L$ terms, we denote it by $I$.

Let us present the equations of powers form item 3) in form of  two equations

\begin{equation}\label{split}
 \begin{cases}
   {{x}_{0}}+{{x}_{1}}+...+{{x}_{L}}=k,
   \\
   {{y}_{0}}+{{y}_{2}}+...+{{y}_{L}}=n,
    \end{cases}
\end{equation}

where ${{x}_{0}}=x+F, \, {x}_{i}=u_i,\,{{y}_{1}}=v_i$ are defined in the item 3),
we count coinciding of columns (pairs) of form ${{x}_{i}}={{x}_{j}},\,\,{{y}_{i}}={{y}_{j}}$.

Let $U$ is set of all splits of form (\ref{split}) that starts from $i=1$
\[\left| U-\bigcup\limits_{\begin{smallmatrix}
 i,j=1 \\
 i<j
\end{smallmatrix}}^{L}{{{M}_{ij}}} \right|\le \left| U \right|-\sum\limits_{i=1}^{L-1}{\sum\limits_{j=i+1}^{L}{\left| {{M}_{ij}} \right|}}+\sum\limits_{i=1}^{L-1}{\sum\limits_{\begin{smallmatrix}
 \,\,\,\,\,\,\,\,\,\,{{j}_{1}}={{i}_{1}}+1 \\
 ({{i}_{1}},{{j}_{1}})\ne ({{i}_{2}},{{j}_{2}})
\end{smallmatrix}}^{L}{\sum\limits_{{{i}_{2}}=1}^{L-1}{\sum\limits_{{{j}_{2}}={{i}_{2}}+1}^{{}}{\left| {{M}_{{{i}_{1}}{{j}_{1}}}}\cap {{M}_{{{i}_{2}}j}}_{_{2}} \right|}}}}-...\]

For each pair of powers $\left( i,j \right):\,\,\,\left| {{M}_{ij}} \right|=\sum\limits_{k=0}^{[\frac{k}{2}]}{\sum\limits_{l=0}^{[\frac{n}{2}]}{C_{x-2k+m-3}^{m-3}C_{y-2l+m-3}^{m-3}}}$, where $k,\,l$  are values of powers of  2 and 3 in pairs where these values  $k,\,l$  are coincided.  Therefore to count exact value we should to substitute 	cardinal of intersection.
Further we should count power set of intersections $\left| {{M}_{ij}}\cap {{M}_{pq}} \right|$ by the similar principle.
 The prove is completed.
\end{proof}
\begin{corollary}
The upper bound of the preimages of $m$ in $\varphi (t)=m$ is
$$\frac{C_{k+L-F-y-1}^{L-1} - I}{L!}X.$$
 \end{corollary}
This bound is upper because formula (\ref{CC}) count all numbers of form ${{p}_{i}}={{2}^{{{u}_{i}}}}{{3}^{{{v}_{i}}}}+1$ but not all of them are prime.
\begin{example}
Number of the preimages for $m=2$ is 3 because $\varphi(3)=\varphi(4)=\varphi(6)=2 $ this preimages belongs to first case. Also $ \varphi(5) = \varphi(12)=\varphi(8)=\varphi(10)=2^2 $ this preimages belongs to 3) and 1) cases.
\end{example}

The calculation shows that cardinal of preimages set depends of $h = 2^{\eta_2(m)}, {\eta_2(m)}>1$ and ${\eta_2(m)}, \, {\eta_3(m)>1}$ for investigation this dependence we propose new metric, which we dedicate to the Academic Pratzyvyty M.

$$\rho (h,g)=\left| ln\frac{h(1-g)}{g(1-h)} \right|, \,\, h,g>1,$$
where in general case instead 1 it can be $c=const$ and $h,g>c>0$ (or $c >h,g>0$).
We shortly prove inequality of triangle. If all logarithm not negative then
$ln(\frac{x(1-y)}{y(1-x)}) \le ln(\frac{x(1-z)}{z(1-x)}) + ln(\frac{z(1-y)}{y(1-z)})$
since logarithm is monotonic and continual function, then, after adding the logarithms in left part,  we proceed to
$\frac{x(1-y)}{y(1-x)} \le \frac{x(1-z) z(1-y)}{z(1-x) y(1-z)}=\frac{x(1-y)}{y(1-x)}$. Therefore we prove the equality in this case.
Other one difficult case if one of logarithms, for instance $ln(\frac{z(1-y)}{y(1-z)})$, is negative then $|ln(\frac{x(1-y)}{y(1-x)})|= |ln(\frac{x(1-z)}{z(1-x)}\frac{z(1-y)}{y(1-z)})| =| ln(\frac{x(1-z)}{z(1-x)}) + ln(\frac{z(1-y)}{y(1-z)})|  \le | ln(\frac{x(1-z)}{z(1-x)})| + |ln(\frac{z(1-y)}{y(1-z)})|.$
Thus, inequality holds. Property $\rho(g,h)= \rho(h,g)$ is true due to modulo and property of logarithm.

We will describe this dependence of $g= 2^{\eta_2(m)}, h =3^{\eta_3(m)}$ in the next work.

\end{section}

\section{Conclusion}
The analytic expression for exact multiplicity of inverses for $m = {{2}^{{{2}^{n}}+a}}$, where $a\in N$, $a<{{2}^{n}}$ and $ \varphi(t)=m$ was obtained. As it turned out, it depends on the number of prime numbers Fermat.
Upper and lower bounds for arbitrary number $n$ were obtained by us. The method of constructing of preimages set for obtained by us lower bound was proposed by us.

%a) If $\varphi ({{2}^{y}}{{3}^{x}})$, then  $y\in \{1,...,k+1\}$,   $x\in N\cup \{0\}$.

%b) If    and  the following conditions holds:

%$\left\{ \begin{matrix}
%   \left( x-1 \right)+1+~{{\text{ }\!\!\eta\!\!\text{ }}_{2}}\left( {{p}_{1}} \right)=k,  \\
%   y-1+~{{\text{ }\!\!\eta\!\!\text{ }}_{2}}\left( \varphi ({{p}_{1}}) \right)=~n.  \\
%\end{matrix} \right.$

%c)   If  $\varphi ({{2}^{x}}{{3}^{y}}\,\cdot \prod{{{p}_{i}}})$ where ${{p}_{i}}\in P$ such that  $\varphi ({{p}_{i}})={{2}^{{{i}_{1}}}}{{3}^{{{i}_{2}}}}$.

%\section{Conclusion }
%  The commutator width of Sylow 2-subgroups of alternating group ${A_{{2^{k}}}}$, permutation group ${S_{{2^{k}}}}$ and Sylow $p$-subgroups of $Syl_2 A_p^k$ ($Syl_2 S_p^k$) is equal to 1. Commutator width of permutational wreath product $B \wr C_n$, were $B$ is an arbitrary group, was researched.
%  \begin{Conclusion}
%  \end{Conclusion}


\begin{thebibliography}{9}

\bibitem{Luca}
\emph {Michal Kevek, Florian Luca, Lawrence Somer} 17 Lectures on Fermat Numbers: From Number Theory  to Geometry, Springer, CMS Books 9, ISBN 0-387-95332-9.

\bibitem {Rodney} \emph {\it Rodney Coleman} On the image of Euler's totient function.  Journal of Computer mathematics Sci. (2012), Vol.3 (2), P. 185-189.

\bibitem{SkVor}	  \emph {Ruslan Skuratovskii,} "The investigation of Euler's totient function preimages" Sixth International Conference on Analytic Number Theory and Spatial Tessellations. Voronoy Conference" Book of abstracts. P. 37-39. http://conference.imath.kiev.ua/index.php/voronoi/2018/index/index
http://conference.imath.kiev.ua/index.php/voronoi/2018/paper/view/60

\bibitem{Sk} \emph {R. V. Skuratovskii,} {\it Involutive irreducible generating sets and structure of sylow 2-subgroups of alternating groups.} Romai journal. (2017), Vol. 13 Issue 1, P. 117-139.


\bibitem{Ford}	 \emph {Kevin Ford,} The Number of Solutions of  $\varphi(x) = m$  Annals of Mathematics, Second Series, Vol. 150, No. 1    (1999), P. 283-312.


\bibitem{Vinogradov}  \emph {Ivan Vinogradov,}
Elements of Number Theory Dover Publications, 5th ed. 2016. P. 236.


\bibitem{SkAr} \emph {R. V. Skuratovskii,} Structure and minimal generating sets of Sylow 2-subgroups of alternating groups. Source: https://arxiv.org/abs/1702.05784v2

\bibitem{RSA}
\emph { Bakhtiari M., Maarof M. A.} Serious Security Weakness in RSA Cryptosystem // IJCSI — 2012. — Vol. 9, Iss. 1, No 3. — P. 175–178. — ISSN 1694-0814; 1694-0784

\bibitem{Max}

\emph { Max A. Alekseyev.}
Computing the Inverses, their Power Sums,
and Extrema for Euler’s Totient and
Other Multiplicative Functions. Journal of Integer Sequences, Vol. 19 (2016), Article 16.5.2, P. 11-21.


%\bibitem{SkThes4} \emph {R. V. Skuratovskii,} Structure of commutant and centralizer of Sylow 2-subgroups of alternating and symmetric groups, minimal generating sets of $Syl_2 A_n$, its applications in cryptography
% Romanian. Conference CAIM 2017.
%Source: https://www.romai.ro/conferintele$_$romai/caim2017/list$_$of$_$participants.php

%\bibitem{SkThes5} \emph {R. V. Skuratovskii,}
%Structure of commutant and centralizer, minimal generating sets of. Sylow 2-subgroups $Syl_2 A_n$  of alternating and symmetric %groups. International conference in Ukraine, ATA12. (2017).  https://www.imath.kiev.ua/~topology/.../skuratovskiy.pdf


%\bibitem{Dm}
%\emph {U.~Dmitruk, V.~Suschansky,}  Structure of 2-sylow subgroup of alternating group and normalizers of symmetric and alternating group. UMJ. (1981), N. 3,  pp. 304-312.


%\bibitem{Heinek}
%\emph {H. Heineken}  Normal embeddings of p-groups into p-groups. Proc. Edinburgh Math.
%Soc. 35 (1992), pp. 309-314.


\end{thebibliography}
\end{document}